\documentclass{amsart}
\usepackage{amsthm, amsfonts, amssymb}

  \gdef\rtitle{}

\gdef\rauthor{}

\AtBeginDocument{%
    \markboth{\rauthor, \rtitle\hfil}{\hfil\rauthor, \rtitle}%
    }

\theoremstyle{plain}
\newtheorem{thm}{Theorem}[section]
\newtheorem{es}[thm]{Example}

\newtheorem{cor}[thm]{Corollary}
\newtheorem{lem}[thm]{Lemma}
\newtheorem{prop}[thm]{Proposition}

\theoremstyle{definition}

\newtheorem{oss}[thm]{Remarks}
\newtheorem{os}[thm]{Remark}

\def\ara{\operatorname{ara}}
\def\c{\operatorname{c}}
\def\cd{\operatorname{cd}}
\def\height{\operatorname{ht}}
\def\depth{\operatorname{depth}}
\def\sdim{\operatorname{sdim}}
\def\Spec{\operatorname{Spec}}
\def\Proj{\operatorname{Proj}}
\def\init{\operatorname{in}}
\def\Supp{\operatorname{Supp}}
\def\reg{\operatorname{reg}}
\def\mult{\operatorname{e}}
\def\codim{\operatorname{codim}}
\def\LT{\operatorname{LT}}
\def\GL{\operatorname{GL}}

\def\heigth{\operatorname{ht}}

 \let\oldlabel=\label
\def\prellabel{\marginparsep=1em\marginparwidth=44pt
 \def\label##1{\oldlabel{##1}\ifmmode\else\ifinner\else
 \marginpar{{\footnotesize\ \\ \tt
     ##1}}\fi\fi}}

\begin{document}

\title{Gr\"obner deformations, connectedness and cohomological dimension}
\author{Matteo Varbaro}
\address{Dipartimento di Matematica, Universit\'a  di Genova, Via Dodecaneso 35, I-16146 Genova, Italy}
\email{varbaro@dima.unige.it} \subjclass{}
\date{}
\keywords{}
 \maketitle

\begin{abstract}
In this paper we will compare the connectivity dimension $\c(P/I)$
of an ideal $I$ in a polynomial ring $P$ with that of any initial
ideal of $I$. Generalizing a theorem of Kalkbrener and Sturmfels
\cite{ka-st}, we prove that $\c(P/\LT_{\prec}(I)) \geq \min
\{\c(P/I),\dim(P/I) - 1 \}$ for each monomial order $\prec$. As a
corollary we have that every initial complex of a Cohen-Macaulay
ideal is strongly connected.  Our approach  is based on the study
of the cohomological dimension of an ideal $\mathfrak{a}$ in a
noetherian ring $R$  and its relation with  the connectivity
dimension of $R/\mathfrak{a}$. In particular we prove a
generalized version of a  theorem of Grothendieck \cite{SGA2}. As
consequence of these results we  obtain some necessary conditions
for  open subscheme of a projective scheme to be  affine.
\end{abstract}

\markboth{\large Gr\"obner deformations, connectedness and cohomological dimension}{\large Matteo Varbaro}

 \section*{Introduction}
All rings considered in this paper are  commutative with identity.
Moreover, throughout the paper, we use the following notation:
\begin{itemize}
\item[(a)] $R$ is a noetherian ring; \item[(b)] $\mathfrak{a}
\subseteq R$ is an ideal of $R$; \item[(c)] $P=k[x_1, \ldots,
x_n]$ is a polynomial ring in $n$ variables (with $k$ an arbitrary
field); \item[(d)] $I \subseteq P$ is an ideal of $P$.
\end{itemize}

With a slight abuse of  terminology in the following we say that $I$
is Cohen-Macaulay to mean that $P/I$ is a Cohen-Macaulay ring.
Given a monomial order $\prec$ on $P$ we will denote by
$\LT_{\prec}(I)$ the initial ideal of $I$ with respect to $\prec$.
A main theme in Gr\"{o}bner bases theory is to compare 
$I$ and $\LT_{\prec}(I)$. In this direction a theorem, due to Kalkbrener and
Sturmfels (\cite[Theorem 1]{ka-st}),  asserts that if $I$ is a
prime ideal, then $P/\LT_{\prec}(I)$ is equidimensional, solving a
conjecture of Kredel and Weispfenning (see \cite{kr-we}).
Moreover, if $k$ is algebraically closed, Kalkbrener and Sturmfels
proved also that $P/\LT_{\prec}(I)$ is connected in codimension
$1$, opening up a new line of research. In light of these results  it is natural  to ask, for example, weather  $\LT_\prec(I)$ has some special features when $I$ is Cohen-Macaulay.
To answer this question we generalize in Theorem \ref{martina}   the
result  of Kalkbrener and Sturmfels by  comparing the connectivity
dimension of $P/I$  with that of $P/\LT_{\prec}(I)$.
 Our result is characteristic free and holds also for non algebraically closed fields.
As a corollary we obtain:

\vspace{2mm}

\noindent{\bf{Corollary \ref{skinner}.}}\emph{ Assume that $I$ is
Cohen-Macaulay. Then $P/\LT_{\prec}(I)$ is connected in
codimension 1.}

\vspace{2mm}

To prove these statements we follow the  approach of 
Huneke and Taylor \cite[Appendix 1]{hu-ta}, which makes use  of local cohomology
techniques. In particular we generalize some of the ideas contained in the appendix written by Taylor. But of course we have to refine these ideas  to obtain
a stronger result. Among other things, we  need also
Grothendieck's Connectedness Theorem (see Grothendieck
\cite[Expos\'{e} XIII, Th\'{e}or\`{e}me 2.1]{SGA2} or Brodmann and
Sharp \cite[Theorem 19.2.9]{B-S}) which asserts that if $R$ is
local and complete, then
$$\c(R/\mathfrak{a}) \geq \min \{ \c(R),
\sdim R - 1 \} - \ara(\mathfrak{a})$$ where $\c(\cdot)$ stands for the connectivity dimension,  $\sdim( \cdot )$
for the subdimension and $\ara( \cdot )$ for arithmetical rank,
see Section \ref{chap1} for the definitions.

Since $\ara(\mathfrak{a})$ is bounded below by the cohomological
dimension $\cd(R,\mathfrak{a})$ of $\mathfrak{a}$ , it is natural
to ask whether the Connectedness Theorem holds also with
$\ara(\mathfrak{a})$ replaced by $\cd(R,\mathfrak{a})$. We
prove in  Theorem \ref{bart} that this is indeed the case.   As a
corollary we will recover a theorem of Hochster and Huneke
\cite[Theorem 3.3]{ho-hu}.\\
Theorem \ref{bart} also appears in the paper of Divaani-Aazar,
Naghipour and Tousi \cite[Theorem 2.8]{D-N-T}. However, when we wrote this paper, we were not aware of their result. We illustrate a relevant error in \cite[Theorem 3.4]{D-N-T} in Remark \ref{patty}.

In Subsection \ref{subchap1.2} we  present versions of our results
for positively graded $k$-algebras (see Theorem \ref{winchester}
and Corollary \ref{marge}), and for local rings satisfying Serre's condition $S_2$ (Proposition \ref{serre}).

In Subsection \ref{subchap1.3} we obtain some results in the
context of projective schemes over a field, studying the
cohomological dimension of their open subschemes. In particular,
we give some new necessary conditions for the affineness of these open
subschemes. To this aim, we  use the results of Subsection
\ref{subchap1.2} and the Serre-Grothendieck correspondence.

As a consequence of the main result we establish the Eisenbud-Goto conjecture for a new class of ideals (Remark \ref{eisgoto}), those which do not contain a linear form, are connected in codimension 1 and have a radical initial ideal.

Finally, in the last subsection, we generalize and strengthen a result of Hartshorne (\cite{hartshorne}), which asserts that a Cohen-Macaulay local ring is connected in codimension 1 (Proposition \ref{barney} and Corollary \ref{smithers}).

This paper is an outcome of the author's master thesis written
under the supervision of Aldo Conca.  We thank him for many
helpful suggestions and conversations.

\section{On connectivity and cohomological dimension}
\label{chap1}

In this section we  use some techniques of local cohomology: for
the basic definitions, properties and results consult
Grothendieck's lectures \cite{gro} or \cite{B-S}.

For an $R$-module $M$ let $H_{\mathfrak{a}}^i(M)$, $ i \in
\mathbb{N}$, denote the $i$-th local cohomology module of $M$ with
respect to $\mathfrak{a}$. An interesting integer related to these
local cohomology modules is
\[ \cd(M, \mathfrak{a}):=\sup \{i \in \mathbb{N}: H_{\mathfrak{a}}^i(M) \neq 0 \}. \]
called the cohomological dimension of $\mathfrak{a}$ with respect
to $M$.

We have the bounds
\[\height_M(\mathfrak{a}) \leq \cd(M,\mathfrak{a}) \leq \dim M \]
where $\height_M(\mathfrak{a}):= \min \{\dim M_{\wp}: \wp
\supseteq \mathfrak{a} \}$.

Moreover, it is well known that, for all $R$-modules $M$, we have
\[ \cd(R,\mathfrak{a}) \geq \cd(M,\mathfrak{a}).\]
Hence we  call $\cd(R, \mathfrak{a})$ the cohomological dimension
of $\mathfrak{a}$.

A numerical invariant of $\mathfrak{a}$ related to its
cohomological dimension is
\[ \ara(\mathfrak{a}) :=  \min \{ r \in \mathbb{N}: \mbox{ exist } f_1, \ldots, f_r \in R \mbox{ such that } \sqrt{\mathfrak{a}} = \sqrt{(f_1, \ldots, f_r)} \} \]
called the arithmetical rank of $\mathfrak{a}$; we have
\[\ara(\mathfrak{a}) \geq \cd(R,\mathfrak{a}). \]

Let $\mathfrak{b}$ be an ideal of $R$, and $x \in R$ an element of
$R$. There are two interesting exact sequences: the first is the
Mayer-Vietoris sequence
\begin{gather}
\ldots \longrightarrow H_{\mathfrak{a} \cap \mathfrak{b}}^{i-1}(M)
\longrightarrow H_{\mathfrak{a}+\mathfrak{b}}^i(M)
\longrightarrow H_{\mathfrak{a}}^i(M) \oplus H_{\mathfrak{b}}^i(M) \nonumber\\
\longrightarrow H_{\mathfrak{a} \cap \mathfrak{b}}^i(M)
\longrightarrow H_{\mathfrak{a}+\mathfrak{b}}^{i+1}(M)
\longrightarrow \ldots
\end{gather}
and the second is
\begin{gather}
\ldots \longrightarrow H_{\mathfrak{a}}^{i-1}(M_x) \longrightarrow H_{\mathfrak{a}+(x)}^i(M) \longrightarrow H_\mathfrak{a}^i(M) \nonumber\\
\longrightarrow H_{\mathfrak{a}}^i(M_x) \longrightarrow
H_{\mathfrak{a}+(x)}^{i+1}(M) \longrightarrow \ldots
\end{gather}

As we have anticipated, we  divide this section in three
subsections: in the first subsection  we  prove the stronger version of
Grothendieck's result; in the second subsection we  analyze this result in
more concrete cases, for example when $R$ is a positively graded
$k$-algebra; in the third subsection we gives the previous results in
the language of algebraic geometry.

\subsection{A stronger version of the Connectedness Theorem}
\label{subchap1.1}

 We begin by reviewing the definition of connectivity dimension of a ring.
Let $T$ be a noetherian topological space; the connectivity
dimension $\c(T)$ of $T$ is defined as the integer:
\[ \c(T):= \min \{ \dim Z :
Z \subseteq T, Z \mbox{ is closed and } T\backslash Z \mbox{ is
disconnected} \}
\]
with the convention that the emptyset is disconnected of dimension
$-1$. If, for a positive integer $d$, \ $\c(T) \geq \dim(T)-d$ we
 say that $T$ is connected in codimension $d$. Notice that
this definition is slightly different from that given in \cite{hartshorne}; however in the case which we
examine in this paper, thanks to the fact that we deal with catenary rings, the two notions are
the same.\\
For an $R$-module $M$, we write $\c(M)$ instead of $\c(\Supp(M))$.
For more details about this definition we refer to \cite[Chapter
19]{B-S}.

A notion related to connectivity dimension is the subdimension,
$\sdim T$, of a non-empty noetherian topological space $T$: it is
defined as the minimum of the dimensions of the irreducible
components of $T$. Again, for an $R$-module $M$, we write $\sdim
M$ instead of $\sdim(\Supp(M))$.

\begin{os}\label{hilbert}

We state an elementary result which better explains the concept of
connectivity dimension.

For a noetherian topological space $T$, the following are
equivalent:
\begin{itemize}
\item[(1)]$\c(T) \geq d$; \item[(2)]for each $T'$ and $T''$, irreducible
components of $T$, there exists a sequence $T'=T_0, T_1, \ldots,
T_r=T''$ such that $T_i$ is an irreducible component of $T$ for
all $i=0, \ldots, r$ and $\dim (T_j \cap T_{j-1}) \geq d$ for all
$j= 1, \ldots, r$.
\end{itemize}
The condition in (2) is the characterization of connectivity
dimension used in \cite{ka-st}.

\end{os}

The Connectedness Theorem, whose a proof can be found in
\cite[Expos\'{e} XIII, Th\'{e}or\`{e}me 2.1]{SGA2} or in
\cite[Theorem 19.2.9]{B-S}, follows:

\begin{thm}\label{palladineve}(Grothendieck's Connectedness
Theorem). Let $(R,\mathfrak{m})$ be complete and local. Then
\[ \c(R/\mathfrak{a}) \geq \min \{ \c(R), \sdim R - 1 \} - \ara(\mathfrak{a}) \]
\end{thm}

So it is natural ask weather the inequality of the above theorem still hold with $\ara(\mathfrak{a})$ replaced by $\cd(R,\mathfrak{a})$.
As we  show below in  Theorem \ref{bart},  the answer to the above question
is affirmative.  To prove Theorem \ref{bart}  we
follow the lines  of the proof of  \cite[Theorem 19.2.9]{B-S},
underlining the necessary changes.

We first prove a proposition which relates the cohomological
dimension of the intersection of two ideals with the dimension of
their sum (corresponding to \cite[Proposition 19.2.7]{B-S}).

\begin{prop}\label{homer}

Let $(R,\mathfrak{m})$ be a complete local domain and let
$\mathfrak{b}$ be an ideal of $R$. Assume that $\min \{ \dim
R/\mathfrak{a}, \dim R/\mathfrak{b} \}
>\dim R/(\mathfrak{a}+\mathfrak{b})$. Then
\[ \cd(R, \mathfrak{a} \cap \mathfrak{b}) \geq \dim R - \dim R/(\mathfrak{a}+\mathfrak{b}) - 1 \]

\end{prop}

\begin{proof}

Set $n:= \dim R$ and $d:= \dim R/(\mathfrak{a}+\mathfrak{b})$, and
we induct upon $d$. If $d=0$ we consider the Mayer-Vietoris
sequence, Equation (1),
\[ \ldots \longrightarrow H_{\mathfrak{a} \cap \mathfrak{b}}^{n-1}(R)  \longrightarrow H_{\mathfrak{a}+\mathfrak{b}}^n(R) \longrightarrow H_{\mathfrak{a}}^n(R) \oplus H_{\mathfrak{b}}^n(R) \longrightarrow \ldots \]
Since $R$ is a complete domain and since $\dim R/\mathfrak{a} > 0$
and $\dim R/\mathfrak{b} > 0$ we can use the
Hartshorne-Lichtenbaum theorem (see \cite[Theorem 8.2.1]{B-S}) to
deduce that $H_{\mathfrak{a}}^n(R)=H_{\mathfrak{b}}^n(R)=0$.
Moreover, since $d=0$, \ $\sqrt{\mathfrak{a}+\mathfrak{b}} =
\mathfrak{m}$, so we have $H_{\mathfrak{a}+\mathfrak{b}}^n(R) \neq
0$ (see \cite[Theorem 6.1.4]{B-S}). Then we must have $H_{\mathfrak{a} \cap
\mathfrak{b}}^{n-1}(R) \neq 0$, hence $\cd(R, \mathfrak{a}\cap
\mathfrak{b}) \geq n-1$.

Let, now, $d>0$. The difference between our proof and that in \cite[Proposition 19.2.7]{B-S} is in this step.\\
We can choose $x \in \mathfrak{m}$, $x$ not in any minimal prime
of $\mathfrak{a}, \mathfrak{b}$ and $\mathfrak{a}+\mathfrak{b}$.
Then let $\mathfrak{a}':=\mathfrak{a}+(x)$ and
$\mathfrak{b}':=\mathfrak{b}+(x)$. From the choice of $x$ follows
that $\dim R/(\mathfrak{a}'+\mathfrak{b}')=d-1$, $\dim
R/\mathfrak{a}'=\dim R/\mathfrak{a}-1
>d-1$ and $\dim R/\mathfrak{b}'=\dim R/\mathfrak{b}-1>d-1$; hence by induction we have
$s := \cd(R,\mathfrak{a}' \cap \mathfrak{b}') \geq n-d$.

Since $\sqrt{\mathfrak{a} \cap \mathfrak{b} + (x)}=
\sqrt{\mathfrak{a}' \cap \mathfrak{b}'}$, then $H_{\mathfrak{a}'
\cap \mathfrak{b}'}^i(R)=H_{\mathfrak{a} \cap \mathfrak{b} +
(x)}^i(R)$ for all $i \in \mathbb{N}$, so in this case the exact
sequence in Equation (2) becomes
\[ \ldots \longrightarrow H_{\mathfrak{a} \cap \mathfrak{b}}^{s-1}(R_x) \longrightarrow H_{\mathfrak{a}' \cap \mathfrak{b}'}^s(R)
\longrightarrow H_{\mathfrak{a} \cap \mathfrak{b}}^s(R)
\longrightarrow \ldots
\]
We have $H_{\mathfrak{a}' \cap \mathfrak{b}'}^s(R) \neq 0$, hence
$H_{\mathfrak{a} \cap \mathfrak{b}}^s(R) \neq 0$ or
$H_{\mathfrak{a} \cap \mathfrak{b}}^{s-1}(R_x) \neq 0$, so
$\cd(R,\mathfrak{a} \cap \mathfrak{b}) \geq s-1 \geq n-d-1$.

\end{proof}

Our goal, now, is to generalize Proposition \ref{homer} to the
case when $R$ is not necessarily a domain.\\
To this purpose we need the following useful lemma.

\begin{lem}\label{milhouse}

Let $T$ be a non-empty noetherian topological space.
\begin{itemize}
\item[(a)] For $r \in \mathbb{N}$, denote by $\mathcal{S}(r)$ the
set of all ordered pairs $(A,B)$ of non-empty subsets of $\{ 1,
\ldots, r \}$ for which $A \cup B = \{ 1, \ldots, r \}$; if $T_1,
\ldots, T_r$ are the irreducible components of $T$, then
\[ \c(T)=\min \{ \dim (( \cup_{i \in A}T_i) \cap( \cup_{j \in B}T_j): (A,B) \in \mathcal{S}(r) \} \]
\item[(b)] $\c(T) \leq \sdim T$. Moreover, if $T$ has finite
dimension, equality holds here if and only if $T$ is irreducible.
\end{itemize}

\end{lem}
A proof of  (a) can be found in  \cite[Lemma 19.1.15]{B-S} and a
proof of (b) can be found in  \cite[Lemma 19.2.2]{B-S}. Now we are
ready to generalize Proposition \ref{homer}.

\begin{prop}\label{lisa}

Let $(R,\mathfrak{m})$ be complete and local and $\mathfrak{b}$ an
ideal of $R$. Assume that $\min \{ \dim R/\mathfrak{a}, \dim
R/\mathfrak{b}\}
>\dim R/(\mathfrak{a}+\mathfrak{b})$. Then
\[ \cd(R, \mathfrak{a} \cap \mathfrak{b})  \geq \min \{ \c(R), \sdim R-1 \} - \dim R/(\mathfrak{a}+\mathfrak{b}) \]

\end{prop}

\begin{proof}

Set $d:= \dim R/(\mathfrak{a}+\mathfrak{b})$, and let $\wp_1,
\ldots \wp_n$ be the minimal primes of $R$.

We first assume that for all $i \in \{ 1, \ldots, n \}$ we have
$\dim R/(\mathfrak{a}+ \wp_i) \leq d$ or $\dim R/(\mathfrak{b}+
\wp_i) \leq d$. After a rearrangement we choose $s:= \sup \{ t \in
\{ 1, \ldots, n \} \mbox{ such}$ $\mbox{that } \dim
R/(\mathfrak{a}+\wp_i) \leq d \mbox{ for all } i \leq t \}$.
Notice that $1 \leq s \leq n - 1$, since we have $\max \{ \dim
R/(\mathfrak{b}+ \wp_k): k \in \{ 1, \ldots , n \} \}= \dim
R/\mathfrak{b}
> d$ and  $\max \{ \dim R/(\mathfrak{a}+
\wp_k): k \in \{ 1, \ldots , n \} \}= \dim R/\mathfrak{a}
> d$; hence $(\{ 1, \ldots, s \} , \{ s+1 ,
\ldots , n \}) \in \mathcal{S}(n)$ (with the notation of Lemma
\ref{milhouse}). We define the ideal of $R$
\[ K:=(\wp_1 \cap \ldots \cap \wp_s) + (\wp_{s+1} \cap \ldots \cap
\wp_n) \] and let $\wp$ be a minimal prime of $K$ such that $\dim
R / \wp= \dim R/K$. By Lemma \ref{milhouse} (a), $\dim R/\wp \geq
\c(R)$. Moreover, since there exist $i \in \{ 1, \ldots, s \}$ and
$j \in \{ s+1, \ldots, n \}$ such that $\wp_i \subseteq \wp$ e
$\wp_j \subseteq \wp$, we have
\[ \dim R/(\mathfrak{a} + \wp) \leq \dim R/(\mathfrak{a} + \wp_i) \leq d \]
\[ \dim R/(\mathfrak{b} + \wp) \leq \dim R/(\mathfrak{b} + \wp_j) \leq d. \]
The injection $R/((\mathfrak{a} + \wp) \cap (\mathfrak{b} +
\wp)) \hookrightarrow R/(\mathfrak{a} + \wp) \oplus
R/(\mathfrak{b} + \wp)$ implies $\dim (R/((\mathfrak{a} + \wp)
\cap (\mathfrak{b} + \wp))) \leq d$ and since $\sqrt{(\mathfrak{a} +
\wp) \cap (\mathfrak{b} + \wp)}= \sqrt{(\mathfrak{a} \cap
\mathfrak{b}) + \wp}$, we have
\[ \dim R/ ((\mathfrak{a} \cap \mathfrak{b}) + \wp) = \dim R/((\mathfrak{a} + \wp) \cap (\mathfrak{b} + \wp)) \leq d.\]
But $R/\wp$ is catenary, (see the book of Matsumura, \cite[Theorem
29.4 (ii)]{matsu}), then
\[ \dim R/ ((\mathfrak{a} \cap \mathfrak{b}) + \wp) = \dim R/ \wp - \height((\mathfrak{a} \cap \mathfrak{b}) + \wp)/\wp) \]
and hence
\[ \height((\mathfrak{a} \cap \mathfrak{b}) + \wp)/\wp) \geq \c(R) - d. \]
So $\cd(R/\wp,((\mathfrak{a} \cap \mathfrak{b}) + \wp)/\wp) \geq
\c(R)-d$, and using the Independence Theorem (\cite[Theorem
4.2.1]{B-S})
\[ \cd(R,\mathfrak{a} \cap
\mathfrak{b})\geq \cd(R/\wp,((\mathfrak{a} \cap \mathfrak{b}) +
\wp)/\wp) \geq \c(R)-d.
\]
Now we discuss the case where there exists $i \in \{ 1, \ldots, n \}$
such that $\dim R/(\mathfrak{a}+\wp_i)>d$ and $\dim
R/(\mathfrak{b}+\wp_i)>d$. We use the Proposition \ref{homer},
considering $R/\wp_i$ as $R$, and $(\mathfrak{a}+\wp_i)/\wp_i$ and
$(\mathfrak{b}+\wp_i)/\wp_i$ as $\mathfrak{a}$ and $\mathfrak{b}$.
Then
\[ d \geq \dim R/\wp_i- \cd( R/\wp_i,( (\mathfrak{a}+\wp_i) \cap (\mathfrak{b}+\wp_i) ) / \wp_i)-1.\]
But $\cd( R/\wp_i,( (\mathfrak{a}+\wp_i) \cap (\mathfrak{b}+\wp_i)
) / \wp_i) = \cd( R/\wp_i,( (\mathfrak{a} \cap \mathfrak{b})+\wp_i
) / \wp_i) \leq \cd( R, \mathfrak{a} \cap \mathfrak{b})$, and
obviously $\dim R/\wp_i \geq \sdim R$, hence
\[ d \geq \sdim R -1 -\cd(R,\mathfrak{a} \cap \mathfrak{b}).\]

\end{proof}

Finally we are able to prove the stronger version of Connectedness
Theorem.

\begin{thm}\label{bart}

Let $(R,\mathfrak{m})$ be complete and local. Then
\[ \c(R/\mathfrak{a}) \geq \min \{ \c(R), \sdim R - 1 \} - \cd(R , \mathfrak{a}) \]

\end{thm}

\begin{proof}

Let $\wp_1, \ldots , \wp_n$ be the minimal primes of
$\mathfrak{a}$ and set $c:=\c(R/\mathfrak{a})$.

If $n=1$, then $c=\dim R/\wp_1$. Let $\wp$ be a minimal prime of
$R$ such that $\wp \subseteq \wp_1$. Using the Independence
Theorem we have $\cd(R, \wp_1) \geq \cd(R/\wp, \wp_1/\wp) \geq
\height(\wp_1/\wp) $. Since $R/ \wp$ is catenary
\[ c = \dim R/\wp_1 = \dim R/\wp - \height(\wp_1/\wp) \geq \sdim R - \cd(R, \wp_1)= \sdim R - \cd(R , \mathfrak{a}). \]
If $n > 1$, let $(A,B) \in \mathcal{S}(n)$ be a pair such that
\[ c= \dim \left( \frac{R}{ (\cap_{i \in A}\wp_i) + (\cap_{j \in B} \wp_j)} \right). \]
Call $\mathcal{J}:= \cap_{i \in A} \wp_i$ and $\mathcal{K}:=
\cap_{j \in B} \wp_j$. Then $\dim R/\mathcal{J}>c$ and $\dim
R/\mathcal{K}>c$. Proposition \ref{lisa} implies
\[ c = \dim R/(\mathcal{J} + \mathcal{K}) \geq \min \{ \c(R) , \sdim R - 1 \} - \cd(R, \mathcal{J} \cap \mathcal{K}) \]
and since $\sqrt{\mathfrak{a}} = \mathcal{J} \cap \mathcal{K}$ the
theorem is proved.

\end{proof}

By Theorem \ref{bart} and the fact that $\ara(\mathfrak{a}) \geq
\cd(R,\mathfrak{a})$ we immediately obtain the Connectedness
Theorem \ref{palladineve}.

Moreover, from Theorem \ref{bart} follows also a theorem, proved
in \cite[Theorem 3.3]{ho-hu}, which generalizes a result of
Faltings given in \cite{faltings}. See also Schenzel
\cite[Corollary 5.10]{schenzel}.

\begin{cor}\label{frink}(Hochster-Huneke). Let $(R, \mathfrak{m})$ be a
complete equidimensional local ring of dimension $d$ such that
$H_{\mathfrak{m}}^d(R)$ is an indecomposable $R$-module. Then
$\c(R/\mathfrak{a}) \geq d - \cd(R,\mathfrak{a}) - 1$. In
particular, if $\cd(R, \mathfrak{a}) \leq d-2$ the punctured
spectrum $\Spec(R/\mathfrak{a}) \setminus \{ \mathfrak{m} \}$ of
$R/ \mathfrak{a}$  is connected.
\end{cor}
\begin{proof}
\cite[Theorem 3.6]{ho-hu} implies that $\c(R) \geq d - 1$, so
the thesis is a consequence of Theorem \ref{bart}. For the last
statement we only have to observe that $\c(\Spec(R/ \mathfrak{a})
\setminus \{\mathfrak{m} \})= \c(R/ \mathfrak{a}) - 1$.
\end{proof}

\begin{os}\label{patty}
In \cite[Theorem 3.4]{D-N-T}  the authors claim  that Corollary
\ref{frink} holds  without the assumption  that
$H_{\mathfrak{m}}^d(R)$ is indecomposable. This is not correct;
indeed the
 converse of Corollary \ref{frink} is true.
That is, if $R$ is a complete equidimensional ring of dimension
$d$ and if $\c(R/\mathfrak{b}) \geq d - \cd(R,\mathfrak{b}) - 1$
holds for all ideals $\mathfrak{b} \subseteq R$, then taking
$\mathfrak{b}=0$ it follows that $R$ is connected in codimension 1.
This implies, by \cite[Theorem 3.6]{ho-hu},  that
$H_{\mathfrak{m}}^d(R)$ is indecomposable.

An explicit  counterexample to \cite[Theorem 3.4]{D-N-T} and to
\cite[Corollary 3.5]{D-N-T}, is given by
$R=k[[x,y,u,v]]/(xu,xv,yu,yv)$, $M=R$ and $\mathfrak{a}$ the zero
ideal. The minimal prime ideals of $R$ are
$(\overline{x},\overline{y})$ and $(\overline{u},\overline{v})$,
so $R$ is a complete equidimensional local ring of dimension $2$.
By part (a) of Lemma \ref{milhouse} we obtain $c(R)=0$, which is a
contradiction to the cited results.
\end{os}

\begin{os}\label{spiderpork}

If $(R,\mathfrak{m})$ is complete and local and $M$ is a finitely
generated $R$-module, we have
\[ \c(M/\mathfrak{a}M) \geq \min \{ \c(M), \sdim M - 1 \} - \cd(M , \mathfrak{a}), \]
by the following argument: we can consider the complete local ring $S:= R/ (0 :_R
M)$. Then we easily have
$\c(M/\mathfrak{a}M)=\c(S/\mathfrak{a}S)$, $\c(M) = \c(S)$ and
$\sdim M = \sdim S$. Moreover, by \cite[Theorem 2.2]{D-N-T}, or
\cite[Lemma 2.1]{schenzel}, we have
\[ \cd(M , \mathfrak{a})=\cd(S , \mathfrak{a})=\cd(S ,
\mathfrak{a}S)\] and thus the result follows from applying Theorem \ref{bart}
to $S$ and $\mathfrak{a}S$.

\end{os}

By Remark \ref{spiderpork} and the part (b) of Lemma
\ref{milhouse} we obtain the following corollary.

\begin{cor} \label{burns}

Let $(R,\mathfrak{m})$ be complete and local, and $M$ a finitely
generated $R$-module. Then
\[ \c(M/\mathfrak{a}M) \geq  \c(M) - \cd(M , \mathfrak{a}) - 1. \]
Moreover, if $M$ has more than one minimal prime ideal, the
inequality is strict.

\end{cor}

\subsection{Non complete case}
\label{subchap1.2}

Up to now, we have obtained a certain understanding of the
connectivity in the spectrum of a noetherian complete local ring.
In order to apply this knowledge to the graded case, we need
two lemmas. We also obtain results on the connectedness for other noetherian local rings (Corollary \ref{willy} and Proposition \ref{serre}).
The reader can find the proof of the first lemma in \cite[Lemma
19.3.1]{B-S}.

\begin{lem}\label{boe}

Assume that $(R,\mathfrak{m})$ is local. We denote with
$\widehat{R}$ the completion of $R$ with respect to
$\mathfrak{m}$. The following hold:
\begin{itemize}
\item[(i)] $\c(R) \geq \c(\widehat{R})$; \item[(ii)] if $\wp
\widehat{R} \in \Spec(\widehat{R})$ for all minimal prime ideals
$\wp$ of $R$, then equality holds in (i);
\end{itemize}

\end{lem}

The reason why we cannot extend Theorem \ref{bart} to non complete local rings is that the inequality (i) in Lemma \ref{boe} may be strict. However for certain rings the above inequality is actually an equality (Corollary \ref{willy} and Theorem \ref{winchester}), and for other rings this problem can be avoided (Proposition \ref{serre})

\begin{cor}\label{willy}

Let $(R, \mathfrak{m})$ be an $r$-dimensional local analytically
irreducible ring (i.e. $\widehat{R}$ is irreducible). Then
\[ \c(R/\mathfrak{a}) \geq  r - \cd(R , \mathfrak{a}) - 1. \]

\end{cor}

\begin{proof}

By point (i) of Lemma \ref{boe}, we have $\c(R/\mathfrak{a}) \geq
\c(\widehat{R}/\mathfrak{a}\widehat{R})$; moreover, by the
Flat Base Change Theorem (see for example \cite[Theorem
4.3.2]{B-S}), $H_{\mathfrak{a}
\widehat{R}}^i(\widehat{R}) \cong H_\mathfrak{a}^i(R) \otimes_{R}
\widehat{R}$ for all $i \in \mathbb{N}$ and since the natural
homomorphism $R \longrightarrow \widehat{R}$ is faithfully flat,
then $\cd(R, \mathfrak{a})= \cd(\widehat{R},
\mathfrak{a}\widehat{R})$. Also, hypotheses imply
$\c(\widehat{R})= \dim(\widehat{R})$, and it is well known that
$\dim(R)= \dim(\widehat{R})$. Hence we conclude  using Corollary
\ref{burns}.

\end{proof}

\begin{prop}\label{serre}
Let $R$ be a $r$-dimensional local ring which is a quotient of a Cohen-Macaulay local ring, and assume that $R$ satisfies Serre's condition $S_2$. Then
\[ \c(R/\mathfrak{a}) \geq  r - \cd(R , \mathfrak{a}) - 1. \]
In particular if $R$ is a Cohen-Macaulay local ring then $\c(R/\mathfrak{a}) \geq  r - \cd(R , \mathfrak{a}) - 1$.
\end{prop}
\begin{proof}
The completion of $R$, $\widehat{R}$, satisfies $S_2$ as well as $R$ (see \cite[Exercise 23.2]{matsu}). Then $\widehat{R}$ is connected in codimension 1 by Proposition \ref{barney}, so, arguing as in the proof of Corollary \ref{willy}, we conclude.
\end{proof}

In the following we say that $R$ is a $R_0$-algebra finitely generated positively graded on $\mathbb{Z}$ if $R=R_0[\xi_1, \ldots ,\xi_r]$ with $\deg(\xi_j)$ a positive integer.
\begin{lem}\label{krusty}

Let $R_0$ be a noetherian ring, and $R$ a $R_0$-algebra finitely
generated positively graded on $\mathbb{Z}$. Let $\mathfrak{m} =
R_+$ denote the irrelevant ideal of $R$, and
$\widehat{R^{\mathfrak{m}}}$ the $\mathfrak{m}$-adic completion of
$R$. If $\wp$ is a graded prime of $R$, then $\wp
\widehat{R^{\mathfrak{m}}}$ is a prime ideal of $\widehat{R^{\mathfrak{m}}}$.

\end{lem}

\begin{proof}


We have only to note that,
if $\wp$ is a graded prime of $R$, then $R/\wp$ is a noetherian
domain positively graded; so, since
$\widehat{R^{\mathfrak{m}}}/\wp\widehat{R^{\mathfrak{m}}} \cong
\widehat{(R/\wp)^{\mathfrak{m}}}$, \cite[Lemma 7.5]{hu-ta} let us conclude.

\end{proof}

Now we prove a version of Theorem \ref{bart} in the case when $R$
is a graded $k$-algebra.

\begin{thm}\label{winchester}

Let $k$ be a field, and let $R$ be a $k$-algebra finitely
generated positively graded on $\mathbb{Z}$; then, if
$\mathfrak{a}$ is graded,

\[ \c(R/\mathfrak{a}) \geq \c(R) - \cd(R, \mathfrak{a}) - 1. \]
Moreover, if $R$ has more than one minimal prime ideal, the
inequality is strict.
\end{thm}

\begin{proof}
Let $\mathfrak{m}$ be the irrelevant ideal of $R$. Using part
(i) of Lemma \ref{boe} we have
\begin{gather} \c(R_{\mathfrak{m}}/
\mathfrak{a} R_{\mathfrak{m}}) \geq
\c(\widehat{R_{\mathfrak{m}}/\mathfrak{a} R_{\mathfrak{m}}}) =
\c(\widehat{R_{\mathfrak{m}}}/\mathfrak{a}
\widehat{R_{\mathfrak{m}}}).
\end{gather}
Moreover, from Lemma \ref{krusty} follows that for all $\wp \in
\Spec(R_{\mathfrak{m}})$ minimal prime of $R_{\mathfrak{m}}$, \
$\wp \widehat{R_{\mathfrak{m}}} \in
\Spec(\widehat{R_{\mathfrak{m}}})$ (since a minimal prime of $R$
is graded, as the reader can see in the book of Bruns and Herzog \cite[Lemma 1.5.6 (b)
(ii)]{B-H}), then we can use part (ii) of Lemma \ref{boe} to
assert
\begin{gather}
\c(\widehat{R_{\mathfrak{m}}}) = \c(R_{\mathfrak{m}}).
\end{gather}
Besides, as in the proof of Corollary \ref{willy}, we have
\begin{gather}
\cd(R_{\mathfrak{m}}, \mathfrak{a} R_{\mathfrak{m}}) =
\cd(\widehat{R_{\mathfrak{m}}}, \mathfrak{a}
\widehat{R_{\mathfrak{m}}}).
\end{gather}
From (3), (4) and (5) and Corollary \ref{burns} follows that
\begin{gather}
\c(R_{\mathfrak{m}}/ \mathfrak{a} R_{\mathfrak{m}}) \geq
\c(R_{\mathfrak{m}}) - \cd(R_{\mathfrak{m}} , \mathfrak{a}
R_{\mathfrak{m}}) - 1.
\end{gather}
All minimal primes of $R$ and $\mathfrak{a}$ are graded, so they
are contained in $\mathfrak{m}$, hence by point (a) of Lemma
\ref{milhouse} $\c(R)= \c(R_{\mathfrak{m}})$ and
$\c(R/\mathfrak{a})= \c(R_{\mathfrak{m}}/\mathfrak{a}
R_{\mathfrak{m}})$. Therefore, since, by the Flat Base Change
Theorem, $\cd(R_{\mathfrak{m}}, \mathfrak{a} R_{\mathfrak{m}})
\leq \cd(R, \mathfrak{a})$, we have
\begin{gather} \c(R/\mathfrak{a}) \geq \c(R) -
\cd(R, \mathfrak{a}) - 1.
\end{gather}
Moreover, if $R$ has more than one minimal prime ideal, also
$R_{\mathfrak{m}}$ and $\widehat{R_{\mathfrak{m}}}$ are such, so
the inequality in (6), and hence that in (7), is strict.
\end{proof}

\begin{os}\label{nelson}

Proceeding in a similar way as in Remark \ref{spiderpork} we can
deduce from Theorem \ref{winchester} the following more general
fact.

Let $k$ be a field, $R$ a $k$-algebra finitely generated
positively graded on $\mathbb{Z}$ and $M$ a $\mathbb{Z}$-graded
finitely generated $R$-module; then, if $\mathfrak{a}$ is graded,
\[ \c(M/\mathfrak{a}M) \geq \c(M) - \cd(M, \mathfrak{a}) - 1. \]
Moreover, if $M$ has more than one minimal prime ideal, the
inequality is strict.

To prove this we only have to note that $0:_R M \subseteq R$ is a
graded ideal (\cite[Lemma 1.5.6]{B-H}).

\end{os}

Remark \ref{nelson} implies easily the following corollary.

\begin{cor}\label{marge}

Let $k$ be a field, $R$ a $k$-algebra finitely generated
positively graded and $M$ a $\mathbb{Z}$-graded finitely generated
$R$-module; then, if $\mathfrak{a}$ is graded,

\[ \c(\Proj(R) \cap \Supp(M/\mathfrak{a}M)) \geq \c(\Proj(R) \cap \Supp(M)) - \cd(M, \mathfrak{a}) - 1. \]
Moreover, if $M$ has more than one minimal prime ideal, the
inequality is strict.
\end{cor}





\subsection{Cohomological dimension of open subschemes of projective schemes}
\label{subchap1.3}

In this Subsection we give a geometric interpretation of the results
obtained in the Subsection \ref{subchap1.2}.

Given a projective scheme $X$ over a field $k$  and an open subscheme $U$, our purpose is to
find necessary conditions for which the cohomological dimension of
$U$ is less than a given integer.

We recall that the cohomological dimension of a noetherian scheme
$X$, written $\cd(X)$, is the smallest integer $r \geq 0$ such
that:
\[ H^i(X,\mathcal{F})=0 \]
for all $i > r$ and for all quasi-coherent sheaves $\mathcal{F}$
on $X$ (the reader can see \cite{hartshorne3} for several results
about the cohomological dimension of algebraic varieties).

By a well known result of Serre, there is a characterization of
noetherian affine schemes in terms of the cohomological dimension:
a noetherian scheme $X$ is affine if and only if $\cd(X)=0$ (see
Hartshorne \cite[Theorem 3.7]{hartshorne1}). Hence, as a
particular case, in this Subsection we  give necessary conditions
for the affineness of an open subscheme of a projective scheme
over $k$. This is an interesting theme in algebraic geometry, and
it was studied from several mathematicians (see for example
Goodman \cite{goodman}, Hartshorne \cite{hartshorne4} or Brenner
\cite{brenner}).

For example, it is well known that, if $X$ is a noetherian
separated scheme, $U \subseteq X$ an affine open subscheme and $Z
= X \setminus U$, then every irreducible component of $Z$ has
codimension less or equal to 1 (see \cite[Proposition
2.4]{brenner} or, for the particular case in which $X$ is a
complete  scheme, \cite[Chapter II, Proposition
3.1]{hartshorne4}).

In light of this result it is natural to ask: what can we say
about the codimension of the intersection of the various
components of $Z$? To answer  this question we  study, considering
a projective scheme $X$ over a field $k$, the connectivity
dimension of $Z$.

Our discussion is based on a well known result, which relates
the cohomology functors of the global sections with the local
cohomology functors.
This result is known as the
Serre-Grothendieck correspondence: let $X$ be a projective
scheme over a field $k$. In this case, $X= \Proj(R)$ where $R$ is
a graded finitely generated $k$-algebra. Let $Z=
\mathcal{V}_+(\mathfrak{a})$ (where $\mathfrak{a}$ is a graded
ideal of $R$), $U = X \setminus Z$, $M$ a graded $R$-module and
$\mathcal{F}= \widetilde{M}$ the associated quasi-coherent sheaf
on $X$. Then there are the isomorphisms

\begin{gather}
\bigoplus_{m \in \mathbb{Z}} H^i(U, \mathcal{F}(m)) \cong
H_{\mathfrak{a}}^{i + 1}(M) \mbox{ for all } i>0.
\end{gather}
The reader can find this result in \cite[Theorem 20.3.15 and
Remarks 20.2.16(iv)]{B-S}.

The following is the main result of this subsection.

\begin{thm}\label{kaprapal}

Let $X$ be a projective scheme over a field $k$, $U \subseteq X$
an open subscheme and $Z=X \setminus U$. If \ $\cd(U) \leq r$,
then $\c(Z) \geq \c(X)-r-2$, where the inequality is strict if $X$
is reducible.

\end{thm}

\begin{proof}
Let $X= \Proj(R)$ with $R$ a graded finitely generated
$k$-algebra, and let $\mathfrak{a}$ be the graded ideal which
determines $Z$. By hypothesis we have $H^i(U, \mathcal{O}_X(m))=0$
for all $i > r$ and for all $m \in \mathbb{Z}$. Then, since
$\mathcal{O}_X= \widetilde{R}$, from the Serre-Grothendieck
correspondence (8) it follows that $\cd(R, \mathfrak{a}) \leq r + 1$.
Hence from Corollary \ref{marge},
\[ \c(Z) \geq \c(X)-r-2. \]
Moreover, again from Corollary \ref{marge}, if $X$ is reducible,
the inequality is strict.
\end{proof}

From Theorem \ref{kaprapal} we can immediately obtain the
following corollaries.

\begin{cor}\label{abrham}
Let $X,U,Z$ be as in Theorem \ref{kaprapal}. If $U$ is affine,
then $\c(Z) \geq \c(X) -2$, where the inequality is strict if $X$
is reducible. In particular, if $X$ is connected in codimension 1
(for example, $X$ irreducible or $X$ Cohen-Macaulay (see Corollary
\ref{smithers})) and $\codim(Z,X)=1$, then $Z$ is connected in
codimension 1.
\end{cor}

\begin{proof}
By the affineness criterion of Serre $\cd(U)=0$, so we conclude by
Theorem \ref{kaprapal}.
\end{proof}



\begin{cor}\label{martin}

Let $X$ be a projective scheme over $k$ of dimension $r$ and
connected in codimension 1. Let $U$ be an open subscheme of $X$
such that $\cd(U) \leq r-2$. Then $X \setminus U$ is connected.

\end{cor}

Corollary \ref{martin} is well known for $X$ irreducible, see Badescu \cite[Theorem
7.6]{badescu}.

\begin{os}
If $X=\mathbb{P}^r$ Hartshorne showed (see \cite[Theorem 3.2 p. 205]{hartshorne4}) that the viceversa to Corollary \ref{martin} holds, i. e. if $\mathbb{P}^r \setminus U$ is connected then $\cd(U)\leq r-2$.  However the viceversa to Theorem \ref{kaprapal} is far from be true also if $X$ is the projective space: for instance let $C$ be a smooth projective curve of positive genus over $\mathbb{C}$, and $Z$ the Segre product $Z=C \times \mathbb{P}^n \subseteq \mathbb{P}^N$. Then $H^1(Z, \mathcal{O}_Z)\neq 0$ by K\"unneth formula for coherent algebraic sheaves (see the paper of  Sampson and Washnitzer \cite[Theorem 1]{sa-wa}), so $\cd(\mathbb{P}^N \setminus Z)\geq N-2$ by \cite[Corollary 7.5]{hartshorne4}. However $Z$ is an irreducible smooth scheme.
\end{os}

\begin{es} Let $R=k[x,y,z,v,w]/(xyw-zvw)$ and
$\mathfrak{a}=(\overline{x}\overline{z},\overline{x}\overline{v},\overline{y}\overline{z},\overline{y}\overline{v})$
$= \wp_1 \cap \wp_2$, where  $\wp_1=(\overline{x},\overline{y})$
and $\wp_2= (\overline{z},\overline{v})$.\\
Furthermore, set $X= \Proj(R)$, $Z=\mathcal{V}_+(\mathfrak{a})$
and $U= X \setminus Z$. Our aim is to prove that $U$ is not
affine.\\
It is clear that $\height(\wp_1)= \height( \wp_2) = 1$, so $U$ may
be affine. However, $X$ is a complete intersection of $\mathbb{P}^4$, so, using Corollary \ref{marge}, $X$ is connected in codimension 1. But $\height(
\wp_1+ \wp_2)=3$, then $c(Z)=0$ by Lemma \ref{milhouse}; so, by
Corollary \ref{abrham} we conclude that $U$ is not affine.
\end{es}





\section{Connectivity of the initial ideal}
\label{chap2}

In this section we prove Corollary \ref{skinner}  given in  the  Introduction. More
generally, we  compare the connectivity dimension of $P/I$ with
the connectivity dimension of $P/ \init_{\omega}(I)$, where
$\init_{\omega}(I)$ denotes the initial ideal with respect a
weight vector $\omega \in (\mathbb{Z}_+)^n$ of $I$.

In the proof given here we  do not need to assume that the field
$k$ is algebraically closed. Moreover, in their paper, Kalkbrener
and Sturmfels first prove the result for the weight vector
$\omega=(1,1, \ldots, 1)$, assuming that $\init_{\omega}(I)$ is a
monomial ideal: one key step in their approach is the Fulton-Hanson
Connectedness Theorem (\cite[Corollary 19.6.8]{B-S}), which forces
them to assume $k$ algebraically closed. Then they use this case
to prove their result for an arbitrary monomial order
(\cite[Theorem 1]{ka-st}). Finally, they complete the proof for
arbitrary weight vectors (\cite[Theorem 2]{ka-st}).

In our proof, instead, we prove directly a more general result
(Theorem \ref{martina}) for arbitrary weight vectors. To this
purpose, as it is clear from the above discussion, we need the
notion of initial ideal with respect to a weight vector.

Let $\omega= (\omega_1, \ldots, \omega_n) \in \mathbb{N}^n$. Given
an element $f \neq 0$ in the polynomial ring $P$, we consider the
polynomial $f(t^{\omega_1}x_1, \ldots, t^{\omega_n}x_n) \in P[t]$,
and we call $\init_{\omega}(f)$ its leading coefficient. Note that
$\init_{\omega}(f) \in P$ is not necessarily a monomial. For an
ideal $I$ of $P$, set
\[ \init_{\omega}(I):= ( \{\init_{\omega}(f) : f \in I, f \neq 0\} ) \]
where $(A)$ denotes the ideal generated by elements of the set $A$.

For a monomial order $\prec$ we  say that $\omega$ represents
$\prec$ for the ideal $I$ if $\LT_{\prec}(I)= \init_{\omega}(I)$.
The reader can find the proof of the following useful result in
the book of Sturmfels (\cite[Proposition 1.11]{sturmfels2}).

\begin{thm}\label{telespallabob}

Given a monomial order $\prec$ in $P$, there exists $\omega \in
(\mathbb{Z}_+)^n$ which represents $\prec$ for $I$.

\end{thm}

In light of Theorem \ref{telespallabob}, to our purpose we can study, given an ideal, its initial ideals with
respect to weight vectors.

Now we need some results about homogenization and dehomogenization
of ideals of a polynomial ring. Many of them are part of the folklore,
however we state them, with our language, for the convenience of
the reader. These topics can be found in \cite{hu-ta} or in the
book of Kreuzer and Robbiano \cite[Chapter 4, Section 3]{K-R}.

Let $\omega \in \mathbb{N}^n$ and $f \in P$: we define the
$\omega$-degree of $f$ the positive integer

\[ \deg_{\omega} f := \max \{ \omega \cdot a : \underline{x}^a \mbox{ is a term of } f \}. \]
We consider the polynomial $^{\omega}\! f \in P[t]$, where $t$ is
an independent variable, defined as:

\[ ^{\omega}\! f(x_1, \ldots, x_n, t) := f(\frac{x_1}{t^{\omega_1}},\ldots, \frac{x_n}{t^{\omega_n}})t^{\deg_{\omega}f}. \]
We  call $^{\omega}\! f$ the $\omega$-homogenization of $f$.

Moreover, we  call the $\omega$-homogenization of $I$ the
following ideal of $P[t]$:
\[ ^{\omega}\! I:=(\{^{\omega}\! f:f \in I\}). \]
Note that $^{\omega}\! I$ is indeed a graded ideal of the
polynomial ring $k[x_1, \ldots , x_n , t]$ with the grading (which
we  call $\omega$-graduation) defined as: $\deg x_i = \omega_i$
for all $i=1, \ldots, n$ and $\deg t = 1$.

We can define an operation of dehomogenization:
\begin{gather*}
\pi: P[t] \longrightarrow P \\
F(x_1, \ldots , x_n , t) \mapsto F(x_1 , \ldots , x_n , 1)
\end{gather*}
Note that $\pi$, in spite of the homogenization's operation, is a
homomorphism of $k$-algebras.

Now we present some easy, but very useful, remarks:

\begin{oss}\label{lennie}

Let $\omega \in \mathbb{N}^n$. Then
\begin{itemize}
\item[(1)] for all $f \in P$ we have $\pi(^{\omega}\! f)=f$;

\item[(2)] let $F \in P[t]$ be an homogeneous polynomial (with
respect to the $\omega$-graduation) such that $F \notin (t)$. Then
$^{\omega}(\pi(F))=F$; moreover, for all $l \in \mathbb{N}$, if
$G=t^l F$ we have $ ^{\omega}(\pi(G)) t^l=G$;

\item[(3)] if $F \in \mbox{ \ }\! \! \! \! ^{\omega}\! I$, then
$\pi(F) \in I$; in fact if $F \in \mbox{ \ }\! \! \! \! ^{\omega}
\! I$ there exist $f_1, \ldots, f_m \in I$ and $r_1, \ldots , r_m
\in P[t]$ such that $ F = \sum_{i=1}^m r_i \mbox{ \ }\! \! \! \!
^{\omega}\! f_i$, and hence by part (1)

\[ \pi(F)=\sum_{i=1}^m \pi(r_i) \pi(^{\omega}\!
f_i)=\sum_{i=1}^m \pi(r_i) f_i \in I; \]

\item[(4)] in$_{\omega}(I)P[t] + (t) = \mbox{ \ }\! \! \! \!
^{\omega}\! I + (t)$. In particular, since $\init_{\omega}(I)P[t]$
is generated by polynomials in $P$, we have $P[t]/(^{\omega}\! I +
(t)) \cong P/\init_{\omega}(I)$.

\end{itemize}

\end{oss}

Now we introduce two elementary but fundamental lemmas

\begin{lem}\label{ralph}

Let $\omega \in \mathbb{N}^n$ and $I$ and $J$ two ideals of $P$.
Then
\begin{itemize}
\item[(1)]$^{\omega}\! (I \cap J)= \mbox{ \ }\! \! \! \!
^{\omega}\! I \cap \mbox{ \ }\! \! \! \! ^{\omega}\! J$;
\item[(2)]$I$ is prime if and only if \ $^{\omega}\! I$ is prime;
\item[(3)]$^{\omega} \! (\sqrt{I})=\sqrt{^{\omega}\! I}$;
\item[(4)]$I=J$ if and only if \ $^{\omega}\! I= \mbox{ \ }\! \!
\! \! ^{\omega}\! J$; \item[(5)]$\wp_1, \ldots, \wp_s$ are the
minimal primes of $I$ if and only if \ $^{\omega}\! \wp_1, \ldots,
^{\omega}\! \wp_s$ are the minimal primes of $^{\omega}\! I$;
\item[(6)]$\dim P/I +1=\dim P[t]/^{\omega}\! I $.
\end{itemize}

\end{lem}

\begin{proof}
For (1), (2) and (3) see \cite[Proposition 4.3.10]{K-R}; for (2) see also \cite[Lemma 7.3, (1)]{hu-ta}

(4). This follows easily from points (1) and (3) of Remarks
\ref{lennie}.

(5). If $\wp_1, \ldots, \wp_s$ are the minimal primes of $I$, then
$\cap_{i=1}^s \wp_i = \sqrt{I}$. So (1) and (3) imply
$\cap_{i=1}^s \mbox{ \ }\! \! \! \! ^{\omega}\! \wp_i =
\sqrt{^{\omega}\! I}$. Then part (2) implies that all
minimal primes of $^{\omega}\! I$ are contained in the set
$\{^{\omega}\! \wp_1, \ldots,^{\omega}\! \wp_s \}$. Moreover, by
point (4) all the primes in this set are minimal for $^{\omega}\!
I$.

Conversely, if $^{\omega}\! \wp_1, \ldots, ^{\omega}\! \wp_s$ are
the minimal primes of $^{\omega}\! I$,  then $\cap_{i=1}^s \mbox{
\ }\! \! \! \! ^{\omega}\! \wp_i = \sqrt{^{\omega}\! I}$. So from
(1) and (3) follows that $^{\omega}\! (\cap_{i=1}^s \wp_i) =
\mbox{ \ }\! \! \! \! ^{\omega}\! \sqrt{I}$; by part (4)  $\cap_{i=1}^s \wp_i = \sqrt{I}$, so using (2) we
have that all the minimal primes of $I$ are contained in the set
$\{ \wp_1, \ldots, \wp_s \}$. Again using (4), the
primes in this set are all minimal for $I$.

(6). This was proven in \cite[Lemma 7.3, (3)]{hu-ta} when $\in_{\omega}(I)$ is a monomial ideal. With a different argument we obtain the general statement: if $\wp_0 \varsubsetneq \wp_1 \varsubsetneq \ldots
\varsubsetneq \wp_d$ is a strictly increasing chain of prime
ideals such that $I \subseteq \wp_0$, then, by (2)  and (4)
$^{\omega}\! \wp_0 \varsubsetneq \mbox{ \ }\! \! \! \! ^{\omega}\!
\wp_1 \varsubsetneq \ldots \varsubsetneq \mbox{ \ }\! \! \! \!
^{\omega}\! \wp_d \varsubsetneq (x_1, \ldots, x_n, t)$ is a
strictly increasing chain of prime ideals such that $^{\omega}\! I
\subseteq \mbox{ \ }\! \! \! \! ^{\omega}\! \wp_0$, so $\dim P[t]/
^{\omega}\! I \geq \dim P/I +1$. Similarly $\height(^{\omega}\!
I)\geq \height(I)$ and we conclude using the fact that  a
polynomial ring over a field is catenary.
\end{proof}

\begin{lem}\label{piccoloaiutantedibabbonatale}

Let $\omega \in \mathbb{N}^n$. Then
\[\c(P[t]/
^{\omega}\! I) \geq \c(P/I) + 1. \]
\end{lem}
\begin{proof}
Let $\wp_1, \ldots, \wp_k$ be the minimal prime ideals of $I$. By Lemma \ref{milhouse} (a), we can choose $(A, B)\in
\mathcal{S}(k)$ such that, setting $\mathcal{J}:= \cap_{i \in A}
\wp_i$ and $\mathcal{K}:= \cap_{j \in B}\wp_j$,
\[ \c (P/I)= \dim P/(\mathcal{J} + \mathcal{K}). \]
From Lemma \ref{ralph} (5) it follows that $^{\omega}\! \wp_1, \ldots
, ^{\omega}\! \wp_k$ are the minimal prime ideals of $^{\omega}\!
I$, and by point (1) of Lemma \ref{ralph} we have $^{\omega}\!
\mathcal{J} = \cap_{i \in A} \mbox{ \ }\! \! \! \! ^{\omega}\!
\wp_i$ and $^{\omega}\! \mathcal{K} = \cap_{j \in B} \mbox{ \ }\!
\! \! \! ^{\omega}\! \wp_j$. Obviously, if $H, L \subseteq P$ are
ideals of $P$, then $^{\omega}\! H + \mbox{ \ }\! \! \! \!
^{\omega}\! L \subseteq \mbox{ \ }\! \! \! \! ^{\omega}\! (H +
L)$, hence,  using Lemma \ref{ralph} (6), $\dim
P[t]/(^{\omega}\! H + ^{\omega}\! L)\geq \dim P[t]/(^{\omega}\! (H
+ L))= \dim P/(H + L)+1$. Hence
\[ \c(P[t]/^{\omega}\! I) \geq \dim P[t]/(^{\omega}\! \mathcal{J} + ^{\omega}\! \mathcal{K}) \geq
\dim  P/(\mathcal{J} + \mathcal{K}) + 1 = \c (P/I) + 1,\]
so we are done.
\end{proof}

Finally, we are able to prove the main result of this paper.

\begin{thm}\label{martina}

Let $\omega \in (\mathbb{Z}_+)^n$. Then
\[ \c (P/\init_{\omega}(I)) \geq \c (P/I) - 1. \]
Moreover, if $I$ has more than one minimal prime ideal, the
inequality is strict.

\end{thm}

\begin{proof}

Note that $P[t]/^{\omega}\! I$ is a finitely generated and
positively graded $k$-algebra and $(^{\omega}\! I +
(t))/^{\omega}\! I \subseteq P[t]/^{\omega}\! I$ is a graded ideal
of $P[t]/^{\omega}\! I$ (considering the $\omega$-grading).
Hence we can use Theorem \ref{winchester}, and deduce that
\begin{gather}
\c (P[t]/(^{\omega}\! I + (t))) \geq \c( P[t]/^{\omega}\! I) - \cd
(P[t]/^{\omega}\! I ,(^{\omega}\! I + (t))/^{\omega}\! I) - 1.
\end{gather}
Obviously $\ara((^{\omega}\! I + (t))/^{\omega}\! I)\leq 1$, so
$\cd ( P[t]/^{\omega}\! I , (^{\omega}\! I + (t))/^{\omega}\! I )
\leq 1$. Hence
\begin{gather}
\c( P[t]/(^{\omega}\! I  + (t))) \geq \c (P[t]/^{\omega}\! I) - 2.
\end{gather}
By Lemma \ref{piccoloaiutantedibabbonatale} we have
$\c(P[t]/^{\omega}\! I) \geq \c(P/I) + 1$, and by the point (4) of
the Remarks \ref{lennie}, we have that $P[t]/(^{\omega}\! I + (t))
\cong P/\init_{\omega}I$, so
\begin{gather}
\c (P/\init_{\omega}(I)) \geq \c (P/I ) - 1.
\end{gather}
Moreover, if $I$ has more than one minimal prime ideals, also
$^{\omega}\! I$ has (point (5) of Lemma \ref{ralph}), then
inequalities in (9), in (10) and in (11) are strict.
\end{proof}

\begin{cor}
(Kalkbrener and Sturmfels). Let $\omega \in (\mathbb{Z}_+)^n$. If
$I$ is a prime ideal, then $P/\init_{\omega}(I)$ is connected in
codimension 1.
\end{cor}

\begin{proof}
If $I$ is prime then $\c(P/I)= \dim P/I$. So Theorem
\ref{martina} implies the statement.
\end{proof}

\begin{os}\label{abu}
Theorem \ref{martina} and Theorem \ref{telespallabob} imply that
if $I$ has more than one minimal prime then for all monomial orders
$\prec$
\[ \c(P/\LT_{\prec}(I)) \geq \c(P/I). \]
In general this inequality is strict. In fact, for all graded
ideals $I \subseteq P$ and for all monomial orders $\prec$ there
exist a non-empty Zariski open set $U \subseteq \GL(n,k)$ and a
Borel-fixed ideal $J \subseteq P$ such that $\LT_{\prec}(g(I))=J$
for all $g \in U$.  The ideal $J$ is called the generic initial
ideal of $I$,  see  Eisenbud \cite[Theorem 15.18, Theorem
15.20]{eisenbud}. It is known that, since $J$ is Borel-fixed,
$\sqrt{J}= (x_1,\dots,x_c)$  where $c$ is the codimension of $I$,
see \cite[Theorem 15.23]{eisenbud}). Hence $\c(P/J)=\dim P/J =\dim
P/I$.  But $I$ can also be chosen in such a way that $\c(P/I)$ is smaller than $\dim P/I$.
\end{os}

\begin{os}

Sometimes, Theorem \ref{martina} can be used to give upper bounds
for the connectivity dimension of an ideal of $P$. In fact, if $B
\subseteq P$ is a monomial ideal, the connectivity dimension of
$P/B$ is simple to calculate, since the minimal prime ideals of
$B$ are easy to find and are generated by variables. So we can use
characterization of Remark \ref{hilbert} to calculate the
connectivity dimension of $P/B$. For
example, if $I$ is a graded ideal such that $\dim(\Proj(P/I)) \geq
1$, and there exists a monomial order $\prec$ such that
$\c(P/\LT_{\prec}(I))=0$, then $\Proj(P/I)$ is disconnected.

\end{os}

\begin{os}\label{eisgoto}

By Theorem \ref{martina} follows that the Eisenbud-Goto conjecture
is true for a certain class of  ideals: in their paper
\cite{eisenbudgoto}, Eisenbud and Goto conjectured that if $\wp
\subseteq P$ is a graded prime ideal which does not contain linear
forms, then
\[ \reg(P/\wp) \leq \mult(P/\wp)-\height(\wp) \]
where $\reg(\cdot)$ means the Castelnuovo-Mumford regularity, and
$\mult(\cdot)$ means the multiplicity. More generally,  the
inequality is expected to hold for radical graded ideals which are
connected in codimension $1$ and  do not contain linear forms. In
his paper \cite[Theorem 0.2]{terai}, Terai proved the conjecture
for (radical, connected in codimension $1$) monomial ideals. It is
well known that, if $I$ is graded, for any monomial order $\prec$
we have $\reg(P/I) \leq \reg(P/ \LT_{\prec}(I))$, $\mult(P/I)=
\mult(P/\LT_{\prec}(I))$ and $\height(I)=
\height(\LT_{\prec}(I))$. Hence  from the above discussion and by
Theorem \ref{martina} we  have that the Eisenbud-Goto conjecture
holds for ideals   which do not contain  linear forms,   are
connected in codimension $1$,  and  have a radical initial ideal.
\end{os}

\subsection{The initial ideal of a Cohen-Macaulay ideal}
\label{subchap2.1}

A result of Hartshorne  \cite{hartshorne} (see also \cite[Theorem
18.12]{eisenbud}), asserts  that a Cohen-Macaulay ring is
connected in codimension 1. Combining this with  Theorem
\ref{martina} it follows that  the initial ideal of a
Cohen-Macaulay ideal is connected in codimension 1. We
generalize Hartshorne's Theorem giving a formula of the connectivity dimension of a local ring as a function of its depth (Proposition \ref{barney}): the proof of this is very similar to the proof of original Hartshorne's Theorem, but the more general version allows us to obtain
a more precise  result (Corollary \ref{skinner}). Moreover we observe that actually a Cohen-Macaulay ring satisfies a stronger condition than to be connected in codimension 1 (Corollary \ref{smithers}). 

We start with the following  lemma.

\begin{lem}\label{maggie}
Let $(R,\mathfrak{m})$ be local. Then
\[ \dim R/\mathfrak{a} \geq \depth(R) - \depth(\mathfrak{a},R). \]
\end{lem}
\begin{proof}
Set $k:= \depth(R)$, $g:=\depth(\mathfrak{a},R)$, and $f_1,
\ldots, f_g \in \mathfrak{a}$ an $R$-sequence; if $J:=(f_1,
\ldots, f_g)$ we must have $\mathfrak{a} \subseteq \cup_{\wp \in
Ass(R/J)} \wp$, so there exists $\wp \in \mbox{Ass}(R/J)$ such that
$\mathfrak{a} \subseteq \wp$.

Obviously $\depth(R/J)= k - g$; moreover,  using \cite[Theorem
17.2]{matsu}, we have $\dim R/ \wp \geq k - g$; but $\mathfrak{a}
\subseteq \wp$ $\Longrightarrow$ $\dim R/\mathfrak{a} \geq \dim R/
\wp \geq k - g$, just end.

\end{proof}

Now we are ready to generalize Hartshorne's result.

\begin{prop}\label{barney}

Let $(R,\mathfrak{m})$ be a catenary local ring. Then
\[ \c(R) \geq \depth(R) - 1. \]
Moreover, suppose that $R$ satisfies $S_k$ Serre's condition, $k\geq 2$. Then $R$ is connected in codimension 1 and for every two minimal primes $\wp$ and $\wp'$ of $R$ and for each prime ideal $\mathcal{P} \supseteq \wp + \wp'$ there exists a sequence $\wp=\wp_1,$ $\ldots ,\wp_s=\wp'$ such that $\dim R/\wp_i+\dim R/\wp_{i+1} \geq \min \{k, \heigth(\mathcal{P})\}-1$ and $\wp_j \subseteq \mathcal{P}$ is a minimal prime of $R$ for each $i=1, \ldots , s-1$ and $j=1, \ldots ,s$.
\end{prop}
\begin{proof}
We suppose that $\c(R) < \depth(R) - 1$, and look for a
contradiction.

Note that $\sdim R \geq \depth(R)$ (\cite[Theorem 17.2]{matsu}).
From this it follows that $\c(R) < \depth(R) - 1$ if and only if
there exist two ideals, $\mathcal{J}$ and $\mathcal{K}$, of $R$,
such that $\mathcal{J} \cap \mathcal{K}$ is nilpotent,
$\sqrt{\mathcal{J}}$ and $\sqrt{\mathcal{K}}$ are incomparable,
and $\dim R/(\mathcal{J} + \mathcal{K}) < \depth(R)-1$. From the
first two conditions, using the theorem of Hartshorne
(\cite[Theorem 18.12]{eisenbud}), it follows that $\depth(\mathcal{J}
+ \mathcal{K}, R) \leq 1$. Then, from Lemma \ref{maggie}, we have
$\dim R/(\mathcal{J} + \mathcal{K}) \geq \depth(R)-1$, which is a
contradiction.

Now suppose that $R$  satisfies $S_2$ condition. By contradiction, as above, let us  suppose there exist two ideals $\mathcal{J}$ and $\mathcal{K}$ of $R$,
such that $\Spec(R)\setminus \mathcal{V}(\mathcal{J}+\mathcal{K}) \subseteq \mathcal{J} \cap \mathcal{K}$,
$\sqrt{\mathcal{J}}$ and $\sqrt{\mathcal{K}}$ are incomparable, and $\dim R/(\mathcal{J} + \mathcal{K}) < \dim R-1$. Then localize at a minimal prime $\wp$ of $\mathcal{J} + \mathcal{K}$: since $\height(\wp)\geq 2$ it follows by the assumption that $\depth(R_{\wp}) \geq 2$. But $\mathcal{V}(\mathcal{J}R_{\wp})$ and $\mathcal{V}(\mathcal{K}R_{\wp})$ provide a disconnection for the punctured spectrum of $R_{\wp}$, so $\c(R_{\wp})=0<\depth(R_{\wp})-1$, contradicting the first part of the statement.  

For the last part of the statement suppose there exist two minimal prime ideals $\wp$ and $\wp'$ of $R$ and a prime ideal $\mathcal{P}\supseteq \wp + \wp'$ for which the condition is not satisfied. Clearly the minimal primes of $R_{\mathcal{P}}$ are the minimal primes of $R$ contained in $\mathcal{P}$, so the conclusion follows by the previous part and by Remark \ref{hilbert}.
\end{proof}

From the above proposition immediately comes the following corollary.

\begin{cor}\label{smithers}

Let $(R,\mathfrak{m})$ be a Cohen-Macaulay local ring. Then, for any two minimal prime ideals  $\wp$ and $\wp'$ of $R$ and for each prime ideal $\mathcal{P} \supseteq \wp + \wp'$ there exists a sequence $\wp=\wp_1,$ $\ldots ,\wp_s=\wp'$ such that $\heigth{\wp_i}+ \heigth{\wp_{i+1}}\leq 1$ and $\wp_j \subseteq \mathcal{P}$ is a minimal prime of $R$ for each $i=1, \ldots , s-1$ and $j=1, \ldots ,s$. In particular $R$ is connected in codimension 1.

\end{cor}

It is easy to show that Proposition \ref{barney} and Corollary \ref{smithers} hold if $R$ is a positively graded $k$-algebra, too. So by Theorem \ref{martina} it
follows immediately the answer to question of the introduction.

\begin{cor}\label{skinner}

Let $\omega \in (\mathbb{Z}_+)^n$ and $I$ a graded ideal. Then
\[ \c
(P/\init_{\omega}(I)) \geq \depth (P/I) - 1.
\]
In particular, if $P/I$ is Cohen-Macaulay, then
$P/\init_{\omega}(I)$ is connected in codimension 1.

\end{cor}

The following is an example due to Conca.

\begin{es}
Consider the graded ideal
\[I = (x_1x_5+x_2x_6+x_4^2,x_1x_4+x_3^2-x_4x_5,x_1^2+x_1x_2) \subseteq \mathbb{C}[x_1, \ldots ,x_6]=:P.\]
One can verify that $I$ is a prime ideal which is  a complete intersection; in particular $P/I$ is a Cohen-Macaulay domain of dimension 3. However the radical of the initial ideal of $I$ with respect to the lexicographical order $\prec$ is
\[\sqrt{LT_{\prec}(I)} = (x_1,x_2,x_3)\cap (x_1,x_3,x_6)\cap (x_1,x_2,x_5)\cap (x_1,x_4,x_5)\]
and, albeit it is connected in codimension 1, $P/\sqrt{LT_{\prec}(I)}$ is not Cohen-Macaulay. This can be seen considering $\wp = (x_1,x_3,x_6)$, $\wp' = (x_1,x_4,x_5)$ and $\mathcal{P} = \wp+\wp'$ $=$ $(x_1,x_3,x_4,x_5,x_6)$, and applying Corollary \ref{smithers}.\\
This example also provides an ideal $I$ for which $\cd(P,I)<\cd(P,LT_{\prec}(I))$: in fact $\cd(P,I)=3$ because $I$ is a complete intersection of height $3$, and $\cd(P,LT_{\prec}(I))=\operatorname{projdim}(P/\sqrt{LT_{\prec}(I)})>3$ where the equality follows by a result of Lyubeznik in \cite{lyubeznik}.
\end{es}

\begin{cor}
Let $I$ be a graded ideal such that $\Proj(P/I)$ is connected (for instance $\depth(P/I) \geq 2$) and $\prec$ a monomial order. Then
\[\depth(P/\sqrt{LT_{\prec}(I)})\geq 2.\]
\end{cor}
\begin{proof}
Theorem \ref{martina} implies that $\Proj(P/LT_{\prec}(I))$ is connected. Then $\cd(P,LT_{\prec}(I))$ $\leq$ $ n-2$ by \cite[Theorem 3.2, p. 205]{hartshorne4}, and \cite{lyubeznik} implies that $\depth(P/\sqrt{LT_{\prec}(I)})\geq 2$.
\end{proof}

\end{document}